\documentclass[11pt,a4paper,reqno]{amsart}
\usepackage{amsmath}
\usepackage{amsfonts}
\usepackage{amssymb}
\usepackage{amscd}
\usepackage{url}
\usepackage{enumerate}
\usepackage[pdftex,bookmarks=true]{hyperref}
\usepackage{color}
\usepackage{graphicx}

\renewcommand\Re{{\operatorname{Re}}}
\renewcommand\Im{{\operatorname{Im}}}

\newcommand\R{{\mathbf{R}}}
\newcommand\C{{\mathbf{C}}}

\renewcommand\P{{\mathbf{P}}}
\newcommand\E{{\mathbf{E}}}

\newcommand\tr{{\operatorname{tr}}}

\newcommand\dist{{\operatorname{dist}}}

\newcommand\F{{\mathbf{F}}}

\newcommand\col{{\mathbf{c}}}

\newcommand\eps{\varepsilon}

\newcommand\Bb{{\mathbf b}}
\newcommand\Bc{{\mathbf c}}

\newcommand\Be{{\mathbf e}}
\newcommand\Bf{{\mathbf f}}
\newcommand\Bg{{\mathbf g}}

\newcommand\Bi{{\mathbf i}}

\newcommand\Br{{\mathbf r}}

\newcommand\Bu{{\mathbf u}}
\newcommand\Bv{{\mathbf v}}

\newcommand\Bx{{\mathbf x}}
\newcommand\By{{\mathbf y}}
\newcommand\Bz{{\mathbf z}}
\newcommand\CE{{\mathcal E}}
\newcommand\CF{{\mathcal F}}
\newcommand\CG{{\mathcal G}}

\newcommand\CN{{\mathcal N}}

\newcommand\N{{\mathbf N}}

\theoremstyle{plain}
  \newtheorem{theorem}[subsection]{Theorem}

  \newtheorem{lemma}[subsection]{Lemma}

  \newtheorem{example}[subsection]{Example}
  
  \newtheorem{remark}[subsection]{Remark}
  
  \newtheorem{claim}[subsection]{Claim}

\theoremstyle{definition}
  \newtheorem{definition}[subsection]{Definition}

\begin{document}

\title[normality of vectors]{Normal vector of a random hyperplane}

\author{Hoi H. Nguyen}
\address{Department of Mathematics, The Ohio State University, Columbus OH 43210}
\address{School of Mathematics, Institute for Advanced Studies, Princeton NJ 08540}

\email{nguyen.1261@math.osu.edu}

\author{Van H. Vu}
\address{Department of Mathematics, Yale University, New Haven CT 06520}
\email{van.vu@yale.edu}

\thanks{H.~Nguyen is supported by NSF grants DMS-1358648, DMS-1128155 and CCF-1412958. V.~Vu is supported by   NSF  grant DMS-1307797  and AFORS grant FA9550-12-1-0083.}

\maketitle

\begin{abstract}
Let $\Bv_1, \dots, \Bv_{n-1}$ be $n-1$ independent vectors in $\R^n$ (or $\C^n$). We study  $\Bx$, the  unit normal vector of the hyperplane spanned by the $\Bv_i$. 
Our main finding is that $\Bx$ resembles a random vector chosen uniformly from the unit sphere, under some randomness assumption on the $\Bv_i$. 

Our result has applications in random matrix theory.  Consider  an $n \times n $ random matrix with iid entries.  We  first prove an exponential bound on the upper tail for the least singular value, 
 improving the earlier linear bound by  Rudelson and Vershynin. Next, we derive optimal delocalization for   the eigenvectors corresponding to eigenvalues of small modulus. 
 
 \end{abstract}

\section{Introduction}

A real random variable $\xi$ is {\it normalized} if it has mean 0 and variance 1. A complex random variable $\xi$ is normalized if 
 $\xi=\frac{1}{\sqrt{2}}\xi_1 + \frac{1}{\sqrt{2}}\Bi \xi_2$, where $\xi_1,\xi_2$ are iid copies of a real normalized  random variable. 

\begin{example}  Some  popular normalized variables 

\begin{itemize} 
\item real standard Gaussian $\Bg_R =\N(0,1)$, or real Bernoulli $\Bb_R$ which takes value $\pm 1$ with probability $1/2$;
\vskip .1in
\item complex standard Gaussian $\Bg_\C = \frac{1}{\sqrt{2}} \Bg_{1,\R} + \frac{1}{\sqrt{2}} \Bi \Bg_{2,\R}$, or complex Bernoulli $\Bb_\C= \frac{1}{\sqrt{2}} \Bb_{1,\R} + \frac{1}{\sqrt{2}} \Bi \Bb_{2,\R}$.
\end{itemize}
\end{example}

Fixed a normalized random variable $\xi$ and consider the  random vector $\Bv = (\xi_1, \dots, \xi_n )$, whose entries are iid copies of $\xi$.  Sample $n-1$ iid copies  $\Bv_1, \dots, \Bv_{n-1} $ of $\Bv$. We would like to study the 
normal vector of the hyperplane spanned by the $\Bv_i$.

In matrix term, we let 
 $A=(a_{ij})_{1\le i \le n-1, 1\le j\le n}$ be a random matrix of size $n-1$ by $n$ where the entries $a_{ij}$ are iid copies of  $\xi$; the $\Bv_i$ are the row vectors of $A$. 
 Let $\Bx=(x_1,\dots,x_n)\in \F^n$ be a unit vector that is orthogonal to the  $\Bv_i$ (Here and later $\F$ is either $\R$ or $\C$, depending on the support of $\xi$.) 
  First note that recent studies in the singularity probability of random non-Hermitian matrices  (see for instance \cite{BVW, RV-LO}) show that under very general conditions on $\xi$, with extremely high probability $A$ has rank $n-1$. In this case  $\Bx$ is uniquely determined  up to the sign $\pm 1$ when $\F=\R$ or by a uniformly chosen rotation $\exp(\Bi \theta)$ when $\F=\C$. Throughout the paper, we use asymptotic notation under the assumption that $n$ tends   to infinity.  In particular, $X = O(Y)$, $X \ll Y$, or $Y \gg X$  means that $|X| \leq CY$ for some fixed $C$. 

When the entries of $A$ are iid standard gaussian $\Bg_\F$, it is not hard to see that $\Bx$ is distributed as a random unit vector sampled according to the Haar measure in $S^{n-1}$ of $\F^n$.  One then deduces the following properties (see for instance \cite{OVW}[Section 2])

\begin{theorem}[Random gaussian vector]  \label{thm:gaussian} 
Let $\Bx$ be a random vector uniformly distributed on the unit sphere $S^{n-1}$.   
Then,

\begin{itemize} 

\item (joint distribution of the coordinates) $\Bx$ can be represented as  
\begin{equation} \label{P1} \Bx := (\frac{\xi_1}{S}, \dots,  \frac{\xi_n }{S}) \end{equation}  where $\xi_i$ are iid standard gaussian $\Bg_\F$, and $S =\sqrt {\sum_{i=1}^n |\xi_i|^2 }$;
\vskip .05in

\item (inner product with a fixed vector)  for any fixed vector $\Bu$ on the unit sphere,  \begin{equation} \label{P2} \sqrt n  \Bx^\ast \Bu \overset{d} {\to} \Bg_\F  ; \end{equation} 
\vskip .05in

\item (the largest coordinate) for any $C > 0$,  with  probability at least $1 - n^{- C} $
\begin{equation} \label{eq:GOEmax}
	\| \Bx \|_{\infty} \leq \sqrt {\frac{8(C+1)^3 \log n}{n} }; 
\end{equation}
\vskip .05in

\item  (the smallest coordinate) for $n \geq 2$, any $0 \leq c < 1$, and any $a>1$,  
\begin{equation} \label{eq:GOEmin}
	\| \Bx \|_{\min}= \min\{|x_1|,\dots, |x_n|\} \geq \frac{c}{a} \frac{1}{n^{3/2}}
\end{equation}
with probability at least $\exp \left( -2c \right) - \exp \left( - \frac{ a^2 - \sqrt{2 a^2 -1} }{2} n \right)$.



\end{itemize} 
\end{theorem}

Motivated by the {\it universality phenomenon} (see, for instance \cite{TVsurvey}), it is natural to ask if these properties are universal, namely that they hold if $\xi$ is non-gaussian. 
Our result confirms this prediction in a  strong sense. They also have applications in the theory of random matrices, which we will discuss after stating the main result. 

Let us introduce some notations. We say that $\xi$ is sub-gaussian if there exists a parameter $K_0>1$ such that for all $t$

\begin{equation}\label{eqn:xi}
\P(|\xi|\ge t)= O(\exp(-\frac{t^2}{K_0})).
\end{equation}

\begin{definition}[Frequent events] Let $\CE$ be an event depending on $n$ (which is assumed to be sufficiently large). 
\begin{itemize}
\item $\CE$ holds asymptotically almost surely if $\P(\CE)=1-o(1)$.
\item $\CE$ holds with high probability if there exists a positive constant $\delta$ such that $\P(\CE)\ge 1 -n^{-\delta}$.  
\item $\CE$ holds with overwhelming probability, and write $P(\CE)=1-n^{-\omega(1)}$, if for any $K>0$, with sufficiently large $n$ $\P(\CE)\ge 1- n^{-K}$.
\end{itemize}
\end{definition}

\begin{theorem}[Main result]\label{thm:linear} Suppose that $a_{ij}$ are iid copies of a normalized sub-gaussian random variable $\xi$, then the followings hold.

\begin{itemize}

\item (the largest coordinate)  There are constants $C, C_1 >0$ such that for any  $m \ge C_1 \log n $ 
\begin{equation}\label{eqn:infty:1}
\P( \|\Bx\|_\infty  \ge   \sqrt {m/n}  ) \le C n^2 \exp (-m/C).
 \end{equation}  In particularly, with overwhelming probability 
 
 $$\| \Bx \| _{\infty} = O(\sqrt { \frac{  \log n}  {n } }).$$

\vskip .05in

\item  (the smallest coordinate) with high probability 
\begin{equation}\label{eqn:infty:2}
\|\Bx \|_{\min}  \ge  \frac{1 }{n^{3/2} \log^{O(1)}n }.
\end{equation}
\vskip .05in

\item (joint distribution of the coordinates) There  exists  a positive constant $c$ such that 
the following holds: for any $d$-tuple $(i_1,\dots,i_m)$, with $d=n^c$, the joint law of the tuple $(\sqrt{n}x_{i_1},\dots,\sqrt{n}x_{i_d})$ is asymptotically independent standard normal. More precisely, there exists a positive constant $c'$ such that for any measurable set $\Omega \in \F^d$,

\begin{equation}\label{eqn:normality}
|\P((\sqrt{n}x_{i_1},\dots,\sqrt{n}x_{i_d}) \in \Omega) - \P(\Bg_{\F,1},\dots,\Bg_{\F,d}) \in \Omega)| \le d^{-c'},
\end{equation}

\noindent where $\Bg_{\F,1},\dots,\Bg_{\F,d}$ are iid standard gaussian.

\item (inner product with a fixed vector) Assume furthermore that $\xi$ is symmetric, then for any fixed vector $\Bu$ on the unit sphere,  
\begin{equation} \label{eqn:innerproduct} 
\sqrt n \Bx^\ast  \Bu \overset{d} {\to} \Bg_\F . 
\end{equation} 
\vskip .05in



\end{itemize}
\end{theorem}

It also follows easily from  \eqref{eqn:infty:1} and \eqref{eqn:normality} that with high probability $\|\Bx\|_\infty = \Theta(\sqrt{\frac{\log n}{n}})$. Indeed, it is clear that with high probability, with $m = n^c$ for some sufficiently small $c$, $\max\{|\Bg_{\F,1}|,\dots,|\Bg_{\F,m}|\} \gg \sqrt{\log m} =\sqrt{c \log n}$. Thus by \eqref{eqn:normality}, with high probability $\max\{|x_1|,\dots,|x_m|\}\gg \sqrt{\frac{\log n}{n}}$. 

Our approach can  be extended to unit vectors orthogonal to the rows of an iid matrices $A$ of size $(n-k)\times n$,  for any fixed $k$ or even $k$ grows slowly with $n$; the details will appear in a later paper. 

As random hyperplanes appear frequently in various areas, including  random matrix theory, high dimensional geometry, statistics, and theoretical computer science,  we expect  that Theorem \ref{thm:linear} will be useful. For the rest of this section, we  discuss two applications. 

\subsection{Tail bound for the  least singular value of a random iid matrix}  Given an $n \times n$ random matrix $M_n(\xi) $  with entries being iid copies of a normalized  variable $\xi$. 
Let  $\sigma_1 \ge \dots \ge \sigma _n \ge 0$ be its singular values. The two extremal $\sigma_1$ and $\sigma_n$ are of special interest, and was  studied by Goldstein and von Neumann, as they 
tried  to analyze the running time of solving a system of random equations $M_n x =b$. 
 
 In \cite{GN},  Goldstein and von Neumann speculated that $\sigma_n$ is of order $n^{-1/2} $, which  turned out to be correct. In particular, 
  $\sqrt n \sigma_n$  tends to a limiting distribution, which was computed explicitly by Edelman in \cite{Edelman}  in the gaussian case.   
  
\begin{theorem} \label{Edelman} 
For any $t\ge 0$ we have 

$$\P(\sigma_n(M_{\Bg_\R}) \le t n^{-1/2}) = \int_{0}^t \frac{1+\sqrt{x}}{2 \sqrt{x}} e^{-x/2 +\sqrt{x}} dx + o(1)$$

as well as 

$$\P(\sigma_n(M_{\Bg_\C}) \le t n^{-1/2}) = \int_{0}^t e^{-x} dx.$$

\end{theorem}
  
In other words, $\P(\sigma_n(M_{\Bg_\R}) \le t n^{-1/2}) =1- e^{-t/2 +\sqrt{t}} + o(1)$ and  $\P(\sigma_n(M_{\Bg_\C}) \le t n^{-1/2}) = 1- e^{-t} $. These distributions have been confirmed to be universal (in the asymptotic sense) by Tao and the second author \cite{TVleast}. 

In applications,  one usually needs large deviation results, which show that the probability that $\sigma_n$ is far from its mean is very small. For the lower bound, Rudelson and Vershyin \cite{RV-LO} proved that for any $t>0$ 
 
\begin{equation}\label{eqn:leastsing:lower}
\P ( \sigma_n \le t n^{-1/2})  \le C t + .999^n, 
\end{equation}
 
 which is sharp up to the constant $C$.  For the upper bound,  in a different paper \cite{RV-upper}, the same authors showed
 
\begin{equation}\label{eqn:leastsing:upper}
\P( \sigma_n \ge t n^{-1/2} )  \le C \frac{\log t}{t} .
\end{equation}

Using Theorem \ref{thm:linear}, we improve this result significantly by proving an exponential tail bound,

 \begin{theorem}[Exponential  upper tail for  the least singular values] \label{cor:upper} Assume that the entries of $M_n=(m_{ij})_{1\le i,j\le n}$ are iid copies of a normalized subgaussian random variable $\xi$ in either $\R$ or $\C$. Then there exist absolute constants $C_1,C_2$ depending on $K_0$ such that
 $$\P( \sigma_n \ge t n^{-1/2} )  \le C_1  \exp( - C_2 t ) . $$
 \end{theorem} 
 
Our proof of Theorem \ref{cor:upper} is totally different from that of \cite{RV-upper}. As showed in the gaussian case, the  exponential bound is sharp,  up to the value of $C_2$.

\subsection{Eigenvectors of random iid matrices.} Our theorem is closely related to (and in fact was motivated by) recent results concerning delocalization and normality of eigenvectors of random matrices.  For random Hermitian matrices, there have been many results achieving almost optimal delocalization of eigenvectors, starting with the work \cite{E3} by Erd\H{o}s et al.  and and continued by Tao et al. and by many others in  \cite{TVuniversality, VW,1,2,3,4,5,6,7,8,9}. Thanks to new universality techniques, one also proved normality of the eigenvectors; see for instance the work \cite{KY} by Knowles and Yin, \cite{TVvector} by Tao and Vu, and \cite{BY} by Bourgade and Yau.

For non-Hermitian random matrix $M_n(\xi)=(m_{ij})_{1\le i,j\le n}$, much less is known. Let $\lambda_1,\dots, \lambda_n$ be the eigenvalues  with $|\lambda_1|\ge \dots \ge |\lambda_n|$. Let $\Bv_1,\dots,\Bv_n$ be the corresponding unit eigenvectors (where $\Bv_i$ are chosen according to the Haar measure from the eigensphere if the corresponding roots are multiple). 
 Recently, Rudelson and Vershynin  \cite{RV-del} proved that with overwhelming probability all of the eigenvectors satisfy 

\begin{equation}\label{eqn:RV:deloc}
\|\Bv_i\|_\infty =O(\frac{\log^{9/2} n}{\sqrt n}).
\end{equation}


 
 By modifying the   proof of Theorem \ref{thm:linear}, we are able  sharpen this bound  for eigenvectors of eigenvalues with small modulus.

\begin{theorem} [Optimal delocalization for small eigenvectors]\label{cor:eigenvectors} Assume that the entries of $M_n=(m_{ij})_{1\le i,j\le n}$ are iid copies of a normalized subgaussian random variable $\xi$ in either $\R$ or $\C$. Then for any fixed $\eps>0$, with overwhelming probability the following holds for any unit eigenvector $\Bx$ corresponding to an eigenvalue $\lambda$ of $A$ with $|\lambda| = O(1)$

$$\|\Bx\|_\infty =O(\sqrt{\frac{\log n}{n}}).$$

\end{theorem} 


We believe  that the individual eigenvector in Theorem \ref{cor:eigenvectors} satisfies the normality property \eqref{eqn:normality}, which 
would imply that the bound $O(\sqrt{\frac{\log n}{n}})$ is optimal up to a multiplicative constant.   Figure \ref{fig:1} below shows that the first coordinate of the eigenvector corresponding to the smallest eigenvalue 
behaves like a gaussian random variable.

\begin{center} 
\begin{figure}[!ht]
   \centerline{\includegraphics[width=0.6\textwidth]{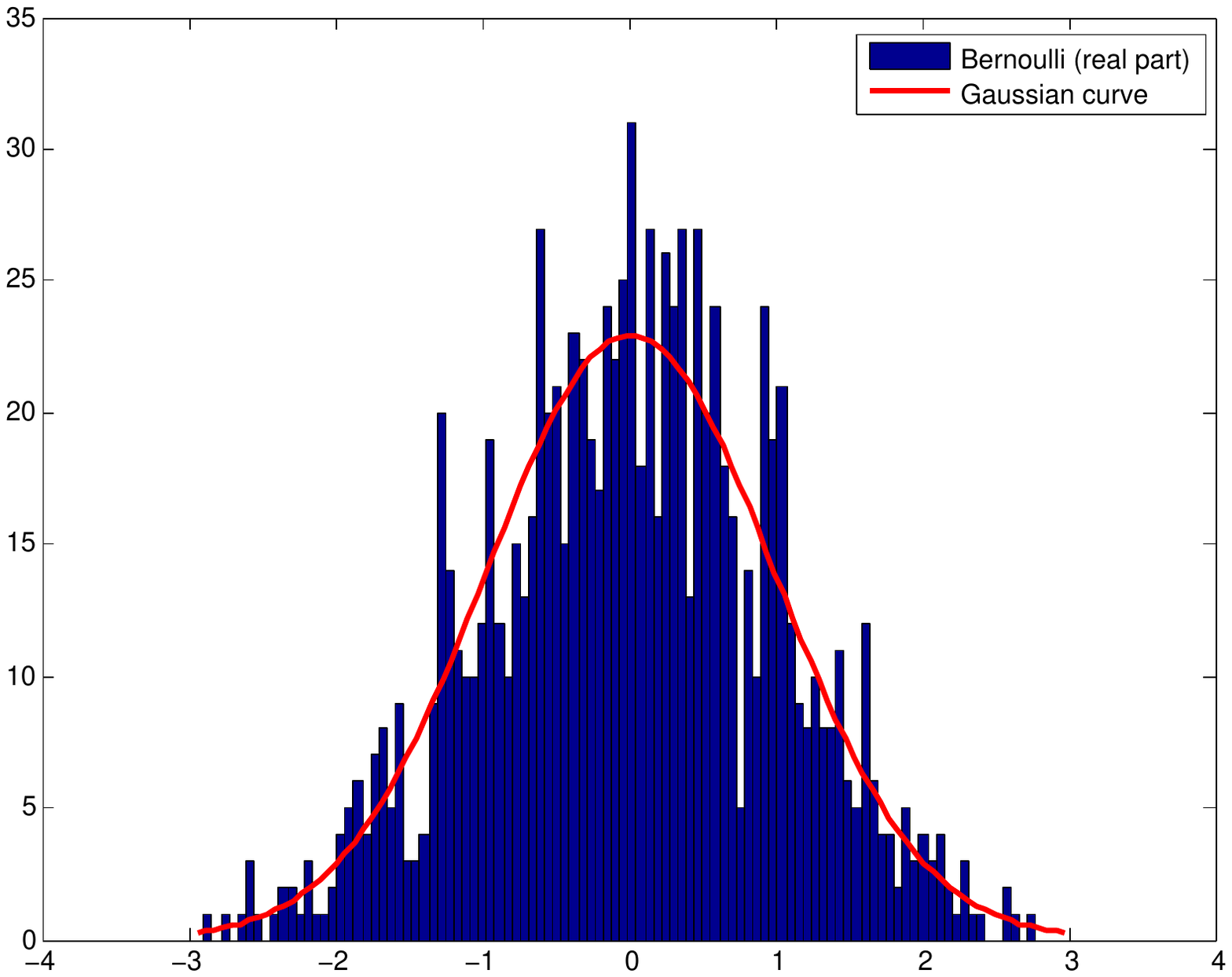} \includegraphics[width=0.6\textwidth]{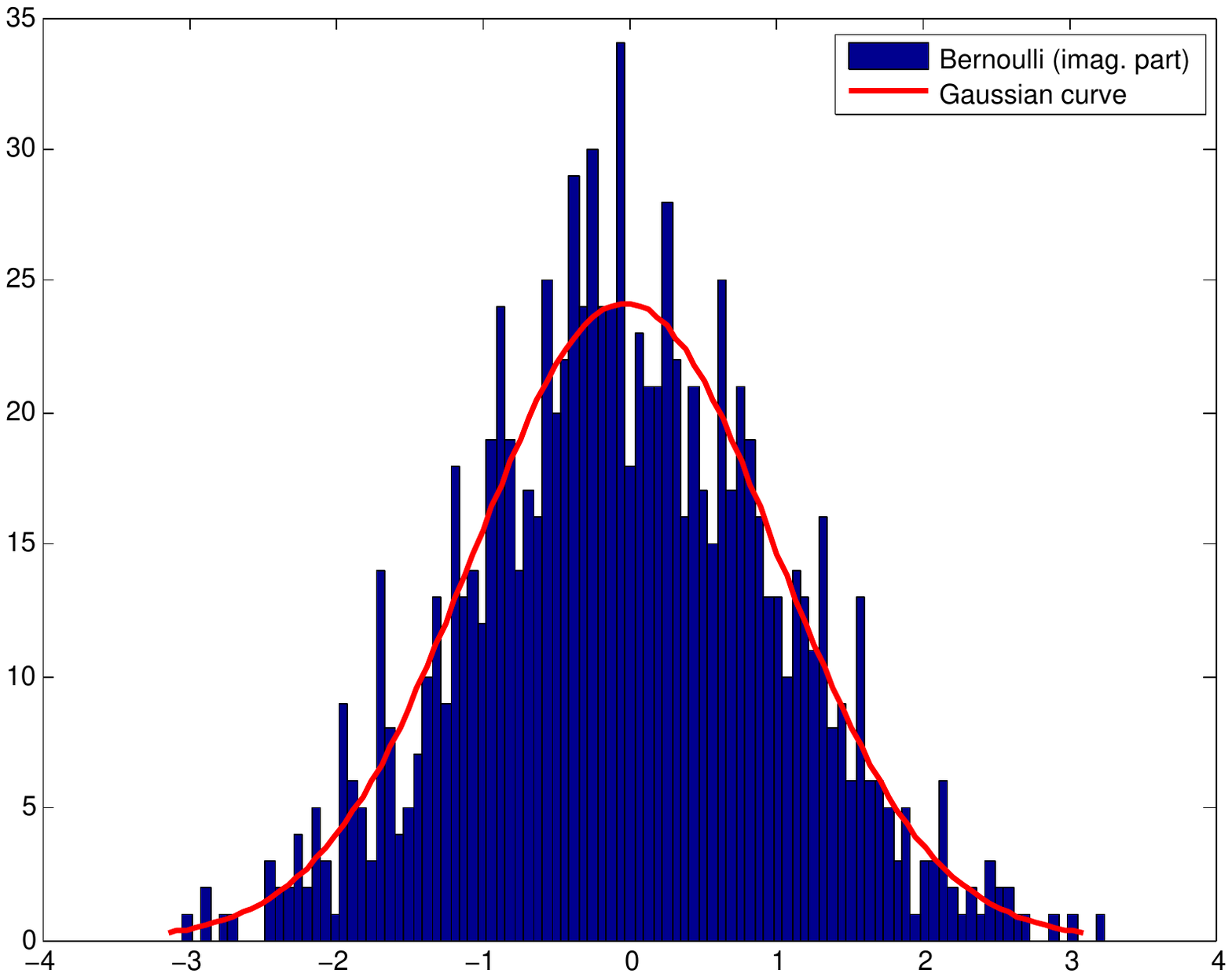} }
   \caption{We sampled 1000 random complex iid Bernoulli matrices of size $n=500$.  The histograms represent the normalized real and imaginary parts  $\sqrt{2n} \Re(\Bv(1))$ and  $\sqrt{2n} \Im(\Bv(1))$ of the first coordinate of the unit eigenvector $\Bv$ associated with the eigenvalue of smallest modulus.} \label{fig:1}
\end{figure}
\end{center}

Finally, let us mention that all of our results  holds (with logarithmic correction) under a weaker assumption that  
the  variable $\xi$  is sub-exponential, namely there are positive constants $C, C'$ and $\alpha$ such that for all $t$ 
$\P ( |\xi| \ge t ) \le C  \exp( - C' t^{\alpha } ) $; see Remark \ref{remark:sub}.

The rest of the paper is organized as follows. After introducing supporting lemmas in Section \ref{section:lemmas}, we will prove  \eqref{eqn:infty:1} and Theorem \ref{cor:eigenvectors} in Section \ref{section:infty}. Section \ref{section:normality} and Section \ref{section:innerproduct}  are devoted to proving \eqref{eqn:normality} and  \eqref{eqn:innerproduct} correspondingly, while \eqref{eqn:infty:2} will be shown in Section \ref{section:non-gap}. Finally, we prove Theorem \ref{cor:upper} in Section \ref{section:singularvalue}.

\section{The lemmas}\label{section:lemmas}

We will use the 
 following well-known concentration result of distances in random non-Hermitian matrices  (see for instance \cite[Lemma 43]{TVuniversality}, \cite[Corollary 2.19]{Tao-RMT} or \cite{VW}). 

\begin{lemma}\label{thm:distance} Let $H$ be a subspace of co-dimension $m$ in $\F^l$ and let $P_H$ be the projection matrix onto the complement $H^\perp$ of $H$. Let $\Bu=(u_1,\dots,u_{l})$ and $\Bv=(v_1,\dots,v_{l})$ be independent random vectors where $u_i,v_i$ are iid copies of an $\F$-normalized sub-gaussian random variable $\xi$. Then the following holds.  

\begin{enumerate}
\item the distance from $\Bu$ to $H$ is well concentrated around its mean,
$$\P\left(\|P_H \Bu\|_2- \sqrt{m}| \ge t \right) \le  \exp(-t^2/K_0^4);$$
\item the correlation $\Bv^T P_H \Bu$ is small,
 $$\P\left( |\Bv^T P_H \Bu| \ge t \right) \le  \exp(-t^2/K_0^4).$$
\end{enumerate}

\end{lemma}


More generally, we have 

\begin{lemma}[Hanson-Wright inequality]\label{lemma:HW} There exists an absolute constant $c$ such that the following holds for any sub-gaussian $\F$-normalized random variable $\xi$ . Let $A$ be a fixed $l\times l$ Hermitian matrix. Consider a random vector $\Bx=(x_1,\dots,x_l)$ where the entries are iid copies of $\xi$. Then

$$\P(|\Bx^\ast A \Bx-\E \Bx^\ast A \Bx|>t) \le 2\exp\Big(-c\min (\frac{t^2}{K_0^4\|A\|^2_{HS}}, \frac{t}{K_0^2\|A\|_2})\Big).$$

In particularly, for any $t>0$

$$\P\left (|\Bx^\ast A \Bx-\E \Bx^\ast A \Bx|> t \|A\|_{HS} \right) \le O\Big(\exp(-c\frac{t^2}{K_0^4}) +\exp(-c\frac{t}{K_0^2})\Big).$$

\end{lemma}

This lemma was first proved by Hanson and Wright in a special case \cite{HW}. The above general  version is due to Rudelson and Vershynin \cite{RV-HW}; see also \cite{VW} for related results which hold (with logarithmic correction) for sub-exponential variables.

\begin{remark} \label{remark:sub} 
As mentioned at the end of the introduction, the results of this paper hold (with logarithmic correction)  for sub-exponential variables. One can achieve this by repeating the proofs, using the results from \cite{VW} (such as \cite[Corollary 1.6]{VW}) 
instead of Lemmas \ref{thm:distance} and \ref{lemma:HW}. We leave the details as an exercise. 
\end{remark} 

The next tool is Berry-Ess\'een theorem for frames, proved by Tao and Vu in \cite{TVleast}. As the statement is  technical, let us first warm the reader up by the classical Berry-Ess\'een theorem.

\begin{lemma}[Berry-Ess\'een theorem]\label{lemma:BE'} Let $v_1,\dots,v_l\in \F$ be real numbers with $\sum_i |v_i|^2=1$ and let $\xi$ be a $\F$-normalized random variable with finite third moment $\E |\xi|^3 < \infty$. Let $S$ denote the random sum

$$S=\sum_i v_i \xi_i,$$

where $\xi_i$ are iid copies of $\xi$. The for any $t\in \F$ we have 

$$\P(|S|\le t) = \P(|\Bg_\F| \le t) + O(\sum_i |v_i|^3),$$

where the implied constant depends on the third moment of $\xi$. In particularly,

$$\P(|S|\le t) = \P(|\Bg_\F| \le t) + O(\max_i |v_i|).$$

\end{lemma}

\begin{lemma}[Berry-Ess\'een theorem for frames]\cite[Proposition D.2]{TVleast}\label{lemma:BE} Let $1\le k\le l$, and let $\xi$ be an $\F$-normalized and have finite third moment. Let $v_1,\dots,v_l\in \F^k$ be a normalized tight frame for $\F^k$, in other words

$$v_1v_1^\ast + \dots + v_n v_n^\ast =I_k,$$

where $I_k$ is the identity matrix on $\F^k$. Let $S\in \F^k$ denote the random variable 

$$S= \xi_1v_1+\dots+\xi_n v_n,$$

where $\xi_1,\dots,\xi_n$ are iid copies of $\xi$.  Similarly, let $G:=(\Bg_1,\dots,\Bg_k)\in \F^k$ be formed from $k$ iid copies of the standard gaussian random variable $\Bg_\F$. Then for any measurable $\Omega \subset F^k$ and for any $\eps= \eps(k,n) >0$ we have

$$\P\Big(G \in \Omega/ \partial_\eps \Omega\Big) - O\Big(k^{5/2} \eps^{-3} \max_j \|v_j\|_\infty \Big) \le P(S\in \Omega) \le \P\Big(G \in \Omega \cup \partial_\eps \Omega\Big) +O\Big(k^{5/2} \eps^{-3} \max_j \|v_j\|_\infty \Big),$$

where $\partial_\eps \Omega$ is the collection of $\Bx\in \F^k$ such that $\dist_{\|.\|_\infty}(\Bx,\partial \Omega)\le \eps$.
\end{lemma}


\section{Treatment for the largest coordinate: proof of \eqref{eqn:infty:1} and Theorem \ref{cor:eigenvectors}}\label{section:infty}

\subsection{Proof of \eqref{eqn:infty:1}}\label{subsection:zero} By a union bound, it suffices to show that for sufficiently large $C$

\begin{equation}\label{eqn:x1}
\P(|x_1| \ll \sqrt{\frac{m}{n}})=1-O\Big(n\exp(-\frac{m}{C})\Big).
\end{equation}

Let $\col_i$, $1\le i \le n$ be the columns of $A$. Because $\sum_{i=2}^n |x_i|^2 \le 1$, among the $(n-1)/m$ subset sums $|x_2|^2+\dots+|x_m|^2, |x_{m+1}|^2+\dots + |x_{2m-1}|^2,\dots, |x_{n-m+2}|^2+\dots+|x_n|^2$, there is a subset sum which is smaller than $m/n$. With a loss of a factor $n/m$ in probability, without loss of generality we will assume that

$$|x_2|^2+\dots+|x_m|^2 \le \frac{m}{n-1}.$$  

Let $H$ be the subspace generated by $c_j, j \ge m+1$.  Let $P_H$ be the
orthogonal  projection from $\F^{n-1}$ onto  $H^\perp$. We view  $P_H$ as a Hermitian matrix of size $(n-1) \times (n-1)$ satisfying  $P_H^2=P_H$. It is known (see for instance \cite{RV-LO,TVleast, BVW}) that with probability 
$1-\exp(-cn)$ we have $\dim(H^\perp)=m-1$, which implies $\tr(P_H)=m-1$.

Recall that by definition,

\begin{equation}\label{eqn:starting}
x_1 \col_1+ x_{2} \col_{2} +\dots +x_{m} \col_{m} + \sum_{i\ge m+1} x_i \col_i=0.
\end{equation}

Applying $P_H$, we have 

$$x_1 P_H \col_1  =  - P_H ( x_2 \col_{2} +\dots +x_m \col_m),$$ 

which implies 
 
 \begin{equation}\label{eqn:S'}
 |x_1|^2 \| P_H \col_1 \|_2 ^2  =  \| \sum_{j=2}^{m} x_j P_H \col_{j} \|_2^2  =  \sum_{ 2\le j_1 \le j_2 \le m } x_{j_1} \bar{x}_{j_2} \col_{j_1}^T P_H \col_{j_2} := \|Q \Bx'\|_2^2, 
 \end{equation}
 where $\Bx'=(x_2,\dots,x_m)$ and $Q \Bx':= \sum_{j=2}^{m} x_j P_H \col_{j}$. We remark that the $x_i$ here are not deterministic but depend on the column vectors $\col_i$. 
 
 As $Q \Bx'$ is linear, and as $|x_2|^2+\dots+|x_m|^2 \le m/(n-1)$, we have 
 
 $$\|Q \Bx'\|_2 \le \sup_{\By \in \F^{m-1}, \|\By\|=1} \|Q\By\|_2 \sqrt{\frac{m}{n-1}}.$$
 
 Thus
 
\begin{equation}\label{eqn:S}
 |x_1|^2 \| P_H \col_1 \|_2 ^2  \le  \sup_{\By \in \F^{m-1}, \|\By\|=1} \| Q \By\|_2^2 \frac{m}{n-1}.
 \end{equation}

 We are going to estimate the operator norm $\|Q\|_2$ basing the randomness of $\col_{j},2\le j\le m$.
 
 \begin{lemma}\label{lemma:Q:norm} There exists a sufficiently large constant $C$ such that
 $$\P_{\Bc_2,\dots,\Bc_m}( \|Q\|_2^2 \ge C m) =O( \exp( -2(m-1))) .$$
 \end{lemma}

Assume Lemma \ref{lemma:Q:norm} for the moment, we can complete the proof of \eqref{eqn:x1} as follows. First, by Lemma \ref{thm:distance},  $\|P_H \col_1\|_2^2 \ge m/2$ with probability at least $1-\exp(-\frac{m}{4K_0^4})$.  We then deduce from \eqref{eqn:S} and from Lemma \ref{lemma:Q:norm} that  

$$\P(|x_1|^2 \gg \frac{m}{n}) \le O\Big( \frac{n}{m}\exp(-\frac{m-1}{4K_0^4}) + \exp(-2(m-1))\Big)  ,$$

completing the proof.

To prove Lemma \ref{lemma:Q:norm}, we first estimate $\|Q \By\|_2$ for any fixed $\By\in S^{m-2}$. We will show

\begin{lemma}\label{lemma:individual} There exists a sufficiently large constant $C$ such that for any fixed $\By\in \F^{m-1}$ with $\|\By\|_2=1$, 
 $$\P_{\Bc_2,\dots,\Bc_m}( \|Q \By \|_2^2 \ge C m) = O( \exp( -4 (m-1) )).$$
\end{lemma}

The deduction of Lemma \ref{lemma:Q:norm} from Lemma \ref{lemma:individual} is standard, we present it here for the sake of completeness.
\begin{proof}(of Lemma \ref{lemma:Q:norm})
Let $\CN$ be a $(1/2)$-net for the set of unit vectors in $\F^{m-1}$. As is well known,  one can assume that $|\CN| \le 4^m$. Applying Lemma \ref{lemma:individual},

$$\P\Big (\exists \By\in \CN,  \|Q \By \|_2^2 \ge 2 m \Big)= O\Big(|\CN|   \exp( -4 (m-1) ) \Big) = O\Big (\exp( -2(m-1)) \Big).$$

Now for any unit vector $\By'$, there exists $\By\in \CN$ such that $\|\By'-\By\|_2 \le 1/2$, and thus by the triangle inequality

$$\|Q \By'\|_2 \le \|Q \By\|_2 + \|Q(\By-\By')\|_2 \le \|Q \By \|_2 + \|Q\|_2/2.$$

This implies that $\|Q\|_2 \le \sup_{\By \in \CN} \|Q \By\|_2 + \|Q\|_2/2$, and hence

$$\|Q\|_2 \le 2  \sup_{\By \in \CN} \|Q \By\|_2.$$
\end{proof}

\begin{proof}(of lemma \ref{lemma:individual})  
 Let $\col$ be the concatenation  of $(\col_{i_1},\dots, \col_{i_{m-1}})$, then $\|Q \By\|_2^2$ can be written as a bilinear form $S =  \col^\ast P \col$ where $P$ is the tensor product of $\By \By^\ast$ and $P_H$, with $\By=(y_1, \dots, y_{m-1})$. 
 By construction,  $P$ consists  of $(m-1)^2$ blocks where the $kl$-th block is  the matrix $y_{k}\bar{y}_{l} P_H$. It thus follows that 
 
 $$\| P \|_2 = \| \By \|_2^2 =1 .$$ 
 
Applying Lemma \ref{lemma:HW} to $S= \Bc^\ast P \Bc$, we have 
  
 $$\P ( |S - \tr P | \ge t ) \le O \Big(\exp( - c\frac{t^2} {K_0^4 \| P\|_{HS}^2 } )  + \exp ( -c\frac{t} { K_0^2\| P \|_2 } )\Big). $$
 
 It is easy to show that 
 
 $$\tr P = (m-1) \sum_{j=0}^{m-1} |y_j|^2 = m-1. $$ 
 
 Taking $t = 4  (c^{-1}+1)K_0^2 (m-1): = \alpha (m-1)$, we obtain 
 
 $$ \P \Big( S \ge (\alpha+1)(m-1) \Big) \le O \Big(\exp ( - 16\frac{ (m-1)^2 } {\| P \|_{HS}^2 } ) + \exp (-4 (m-1)) \Big). $$
 
 To this end, by properties of a tensor product, 
 
 $$\| P \| _{HS}^2=  \| \By \By^T \|_{HS}^2 \| P_H \|_{HS}^2 = m-1 , $$ 
 
 which implies that 
 
 \begin{equation}\label{eqn:S:1} 
 \P \Big( S \ge (\alpha+1) (m-1) \Big) =O\Big( \exp( -4(m-1)) \Big). 
 \end{equation}
 
 \end{proof}

 
 We now turn to the eigenvectors.
 
 \subsection{Proof of Theorem \ref{cor:eigenvectors}} We will be working with the perturbed matrix $M_n-\lambda_0$ where $(M_n-\lambda_0)_{ii} = m_{ii}-\lambda_0, 1\le i\le n$ and $(M_n-\lambda_0)_{ij} =m_{ij}, i\neq j$. By a standard net argument, it suffices to show the following

\begin{theorem}\label{theorem:eiginfty} For any fixed $\lambda_0$ with $|\lambda_0| \le O(1)$, the following holds with overwhelming probability with respect to $M_n$: if $\|(M_n-\lambda_0)\Bx\|_2 \le 1/n$ then $\Bx$ satisfies \eqref{eqn:infty:1}.  
\end{theorem}

Equivalently, we show that for any unit vector $\Bx \in \F^n$ satisfying the condition of Theorem \ref{theorem:eiginfty}, then

\begin{equation}\label{eqn:x1'}
\P(|x_1| \ll \sqrt{\frac{m}{n}})=1-O\Big(n\exp(-\frac{m}{C})\Big).
\end{equation}

We will proceed as in Subsection \ref{subsection:zero} by assuming that $|x_2|^2+\dots+|x_m|^2 \le m/(n-1)$, where instead of $\eqref{eqn:starting}$ we have



\begin{equation}\label{eqn:starting'}
x_1 \col_1+ x_{2} \col_{2} +\dots +x_{m} \col_{m} + \sum_{i\ge m+1\}} x_i \col_i =\Br
\end{equation}

for some vector $\Br$ with norm $\|\Br\|_2 \le 1/n$, where $\col_i$ is the $i$-th column of the matrix $M_n-\lambda_0$. 

Projecting onto $H^\perp$, we obtain

$$ |x_1|^2 \| P_H \col_1 \|_2 ^2  \le  2\| \sum_{j=2}^{m} x_{j} P_H \col_{j} \|_2^2 + 2\|\Br\|_2^2 \le  2\sum_{ 2\le j_1 \le j_2 \le m } x_{j_1} \bar{x}_{j_2} \col_{j_1}^\ast P_H \col_{j_2}  + \frac{2}{n^2}. $$
 
 Note that here as $|\lambda_0|=O(1)$, Lemma \ref{thm:distance} is still effective, which yields  $\|P_H \col_1\|_2^2 \ge m/2$ with probability at least $1-\exp(-\frac{m}{4K_0^4})$.
 
To estimate the right hand side, set $Q (\Bx'): =\sum_{j=2}^{m} x_{j} P_H \col_{j} $.  Similarly to Lemma \ref{lemma:Q:norm}, we will establish
  
 \begin{lemma}\label{lemma:Q:norm'} There exists a sufficiently large constant $C$ such that
 $$\P_{\Bc_2,\dots,\Bc_m}( \|Q\|_2^2 \ge C m) =O( \exp( -2(m-1))) .$$
 \end{lemma}

It is clear that \eqref{eqn:x1'} follows from Lemma \ref{lemma:Q:norm'}. Furthermore, similarly to our treatment in the previous subsection, for this lemma it suffices to show the following analog of Lemma \ref{lemma:individual} for any fixed $\By$.

\begin{lemma}\label{lemma:individual'} There exists a sufficiently large constant $C$ such that for any fixed $\By\in \F^{m-1}$ with $\|\By\|_2=1$, 
 $$\P_{\Bc_2,\dots,\Bc_m}( \|Q \By \|_2^2 \ge C m) = O( \exp( -4 (m-1) )).$$
\end{lemma}

It remains to prove Lemma \ref{lemma:individual'}. Write $\col_{j} = \col_{j}' -\lambda_0 \Bf_{j}$, where $\Bf_{j}$ is a $\{0,1\}$-vector with at most one non-zero entry and $\col_{j}'$ is a random vector of iid entries. Thus

  \begin{align*}\label{eqn:S'}
 \sum_{ 1\le i, j \le m-1 } y_{i} \bar{y}_{j} \col_{i}^\ast P_H \col_{j}  &=  \sum_{ 1\le i,j \le m-1 } y_i \bar{y}_j {\col_{i}'}^\ast P_H \col_{j}' \\
 &+ \lambda_0 \sum_{ 1\le i, j \le m-1 } y_i \bar{y}_j {\col_{i}'}^\ast P_H \Bf_{j} \\
 &+ \lambda_0 \sum_{ 1\le i, j \le m -1} y_{i} \bar{y}_{j} {\Bf_{i}}^\ast P_H \col_{j}'\\
 &+ |\lambda_0|^2 \sum_{ 1\le i,j \le m-1 } y_i \bar{y}_j {\Bf_{i}}^\ast P_H \Bf_{j} \\
 &:= S+ S'+S''+S'''.
  \end{align*}

For $S$, argue similarly as in the proof of Lemma \ref{lemma:individual}, we obtain the following analog of \eqref{eqn:S:1}
 
 $$ \P \Big( S \ge (\alpha+1)(m-1)  \Big) =O( \exp( -4(m-1)). $$
 
 
 
Next, we have
 
\begin{align*}
|S'|= |\lambda_0\sum_{ 1\le i,j \le m } y_{{i}} \bar{y}_{{j}} {\col_{{i}}'}^\ast P_H \Bf_{{j}}|&=|\lambda_0 (\sum_{ 1\le i  \le m -1} y_{{i}} {\col_{{i}}'}^\ast)( \sum_{1\le j \le m-1} y_{j}P_H \Bf_{{j}})|\\
&=| \lambda_0 (\sum_{ 1\le i  \le m-1 } y_{{i}} {\col_{{i}}'}^\ast ) P_H(\sum_{1\le j \le m-1} \bar{y}_{j} \Bf_{{j}})| ,
\end{align*}

Additionally, as $\|P_H\|_2 \le 1$ and $\|\By\|_2=1$, by the properties of $\Bf_i$ the vector $\Bz:= P_H(\sum_{1\le j \le m-1} \bar{y}_{j} \Bf_{{j}})$ has norm at most $\|\Bz\|_2 \le 1$. As such, the subgaussian random variable $(\sum_{ 1\le i  \le m-1 } y_{{i}} {\col_{{i}}'}^\ast)  \Bz$ has variance at most one, and hence

$$\P\Big( |(\sum_{ 1\le i  \le m-1 } y_{{i}} {\col_{{i}}'}^\ast ) \Bz | \ge m-1  \Big) =o\Big( \exp( -4(m-1)\Big). $$

We can argue similarly for $S''$ to obtain the same bound. Finally,  notice that 

$$|S'''| = |\lambda_0|^2 \|P_H( \sum_{1\le j \le m-1} y_{j} \Bf_{{j}})\|_2^2 \le |\lambda_0|^2.$$
  
Putting all the estimates together, we obtain Lemma \ref{lemma:individual'} as long as $|\lambda_0| =O(1)$.

\section{Treatment for the smallest coordinate: proof of \eqref{eqn:infty:2}}\label{section:non-gap}

Let $M$ be the random matrix of size $(n-1) \times (n-1)$ obtained from $A$ by deleting its first column. Set  $\Bx'=(x_2,\dots,x_n)$, we have

 $$ A \Bx = x_1 \col_1 + M \Bx' = 0. $$
 
 As it is known that with  probability at least $1-\exp(-cn)$ the matrix $M$ is invertible; in this case, we can write 
 
 $$x_1 M^{-1} \col_1 = - \Bx' .$$
 
Since
 
 $$ |x_1|^2  \| M^{-1} \col_1\| _2^2 = \| \Bx' \|_2 ^2 = 1 - |x_1| ^2, $$ we obtain 
 
 $$|x_1|^2  = \frac{1} {1 +  \| M^{-1} \col_1 \|_2^2 }  = \frac{1}{ 1 +\sum_{j=1} ^{n-1} \sigma_j^{-2} | \col_1^T \Bu_j | ^2 }, $$ where $\sigma_1\ge \dots \ge \sigma_{n-1}$ are the singular values of $M$ with corresponding left-singular vectors  $\Bu_1,\dots,\Bu_{n-1}$.



We now condition on $M$. By the sub-gaussian property of the entries,  we can easily show that there is a constant $C$ such that with  overwhelming probability (with respect to $\col_1$)

\begin{equation}\label{eqn:c&u}
| \col_1^T \Bu_1 |  \le C \log n \wedge \dots \wedge | \col_1^T \Bu_{n-1} |  \le C \log n.
\end{equation}

We will need the following  estimate

\begin{claim}\label{claim:negative} With respect to $M$ we have

$$\P(\sum_{i=1}^{n-1} \sigma_i^{-2} \le n^3 \log^8 n) \ge 1-\frac{1}{n \log n}.$$
\end{claim}

\begin{proof}(of Claim \ref{claim:negative}) By \eqref{eqn:leastsing:lower}

$$\P(\sigma_{n-1}^{-1} \le n^{3/2} \log^3 n ) \ge 1- \frac{1}{n \log^3 n}.$$ 

Thus by the  union bound

$$\P(\sum_{i=1}^{\log^2 n} \sigma_{n-i}^{-2} \le n^{3} \log^8 n ) \ge 1- \frac{1}{n \log n}.$$ 

For the remaining sum $\sum_{j=1}^{n-\log n-1} \sigma_j^{-2}$, by the Cauchy-interlacing law,

$$\sum_{j=1}^{n-\log^2n-1} \sigma_j^{-2}(M) \le \sum_{j=1}^{n-\log n-1} \sigma_j^{-2}(M'),$$

where $M'$ is obtained from $M$ by deleting its first $\log^2 n$ columns.

On the other hand, by the negative second moment identity (see \cite[Lemma A.4]{TVK}) 


 \begin{equation}\label{eqn:negative1}
 \sum_{j=1}^{n-\log^2 n} \sigma_j^{-2}(M')  = \sum_{j=1}^{n-\log^2n}  d_j^{-2} , 
 \end{equation}
 
where $d_j$ is the distance from the $j$th row of $M'$ to the hyperplane $H_j$ spanned by the remaining rows of $M'$. 
Using  Theorem \ref{thm:distance} and the union bound, we obtain, for  some constant $c$ and with overwhelming probability, that   $d_j \ge c \log n $ simultaneously for all  $1 \le j \le n-\log^2 n$. This  implies that with overwhelming probability with respect to $M$
 
 $$\sum_{j=1}^{n-\log^2n} \sigma_j^{-2}  \ll  \frac{n}{\log^2n }.$$ 
 \end{proof}
 
 Now by  \eqref{eqn:c&u} and Claim \ref{claim:negative}, we have 
 
 $$\P(|x_1| \gg \frac{1}{n^3 \log^{10} n}) \ge 1 - \frac{1}{n \log n}.$$

By the union by,  we have with probability at least $1-\frac{1}{\log n}$,

$$|x_1| \ge \frac{1}{n^3 \log^{10} n} \wedge \dots \wedge |x_n| \ge \frac{1}{n^3 \log^{10} n},$$  proving the desired statement.

\section{Exponential upper tail bounds: proof of Theorem \ref{cor:upper}}\label{section:singularvalue}

  
Using \cite[Theorem 1.3]{TVleast} we can compare $\P( \sigma_n \ge t n^{-1/2} ) $ with  $\P( \sigma_n (\Bg_\F)  \ge t n^{-1/2} ) $, where $\sigma_n (\Bg_\F)$ is the least singular value of an $\F$-normalized gaussian matrix. More precisely, it shows that 
there exists a positive constant $c$ such that
 
 $$\P( \sigma_n \ge t n^{-1/2} )  \le  \P( \sigma_n (\Bg_\F)  \ge t n^{-1/2} )  + n^{-c}.$$

In the complex case, Theorem \ref{Edelman} has   $\P( \sigma_n (\Bg_\C)  \ge t n^{-1/2} )  = \exp( -t)$.
Since  $n^{-c} = \exp (- c \log n)$, this implies the claim for 
 $t \le  C \log n$ for any fixed $C$ and properly chosen constants $C_1, C_2$. 
 
 In the real case, one cannot apply Theorem \ref{Edelman} directly because of the error term is just plainly $o(1)$.  
 However, in \cite{TVleast} Tao and the second author proved that this error term is at most $n^{-c'}$ for some constant $c' >0$.  Thus, one can conclude in the same manner as in the complex case.

 From here we assume $t >  C\log n $, where $C$ is a sufficiently large constant. 
  By the proof of \eqref{eqn:infty:1} of Theorem \ref{thm:linear} (applied for matrices of size $n\times  (n+1)$ instead of $(n-1) \times n$) we have, for all $m \ge C \log n$ that 
 
 $$ \P( |x_1|  \ge m ^{1/2} n^{-1/2} )=O( \exp ( - m )).$$

Equivalently, for all $t = m \ge  C \log n $ 
$$\P( |x_1|  \ge t^{1/2} n^{-1/2} ) =O( \exp (- t )). $$

 One the other hand, similarly to our treatment in Section \ref{section:non-gap}

 $$|x_1| ^2  =\frac{1}{ 1+ \sum_{j=1}^{n} \sigma_j^{-2} ( \col_1^T \Bu_j)^2} , $$ 
 
 where $\sigma_j$ are the singular values of the random square matrix $M_{n}$ formed by the last $n$ columns, $\col_1$ is the first column, and $\Bu_j$ are the corresponding unit eigenvector of 
 $\sigma_j^2$ in  $M_{n} M_{n}^\ast$. 
 
 Thus with probability at least $1- O( \exp (- t ))$ we have
 
 \begin{equation}\label{eqn:weighted:1}
  1+ \sum_{j=1}^{n} \sigma_j^{-2} ( \col_1^T \Bu_j)^2 \ge \frac{n}{t}.
 \end{equation}

 Next, again by following the  argument  in Section \ref{section:non-gap} (using the negative-moment identity  \eqref{eqn:negative1}, the Cauchy-interlacing law, and Theorem \ref{thm:distance}), we can prove

 \begin{claim}\label{claim:non-weight} With probability at least $1 -n\exp (-\frac{t}{K_0^2})$ one has
 
 $$\sum_{j=1} ^{n- 100t}  \sigma_j  ^{-2}  \le  \frac{n}{2t} ,$$
 with $K_0$ from \eqref{eqn:xi}. 
 \end{claim}

 
 To handle the coefficients $|\col^T \Bu_j|$, we use 
  the following concentration result from \cite{VW}.

\begin{lemma}\cite[Lemma 1.2]{VW} \label{lemma:VW} 
Let $\col= (x_1, \dots, x_n)$ be a  random vector where $x_i$ are iid copies of $\xi$. Then there exists a constant $C' >0$ such that the following holds. Let $H$ be a subspace of dimension $d$ with an orthonormal basis $\{\Bu_1,\ldots, \Bu_d\}$. Then for any $0 \le c_1, \dots, c_d \le 1$ and any $s$
\begin{equation*}
\P \left( | \sqrt{\sum_{j=1}^d c_j |\col^T \Bu_j|^2} -\sqrt{\sum_{j=1}^d c_j}  | \ge s \right) \le2 \exp(-C' \frac{s^2}{K_0^4} ).
\end{equation*}
\end{lemma}

\begin{remark}  There is a strong relation between this lemma and Lemma \ref{lemma:HW}. 
First, one can  give a short proof of this lemma using Lemma \ref{lemma:HW}. Second, one can also prove a generalization of  Lemma \ref{lemma:HW}  to sub-exponential variables (with  logarithmic correction) using this lemma.  See Remark \ref{remark:sub}. 
\end{remark}


In particular, by squaring, it follows that
\begin{equation} \label{ineq:VW} 
\P \left( | \sum_{j=1}^d c_j | \col^T \Bu_j|^2 -\sum_{j=1}^d c_j  | \ge 2s\sqrt{\sum_{j=1}^d c_j} + s^2 \right) \le 2 \exp(-C' \frac{s^2}{K_0^4} ).
\end{equation}

Next,  Lemma \ref{lemma:VW},  applied to  $\sum_{j=1}^{n} \sigma_j^{-2} ( \col_1^T \Bu_j)^2$ (with $c_j =\frac{\sigma_j^{-2}}{\sigma_{n}^{-2}}$ and $s=t^{1/2}$),  implies that 

$$\P\left(|\sum_{j=1}^{n} \sigma_j^{-2} ( \col_1^T \Bu_j)^2 - \sum_{j=1}^{n-1} \sigma_j^{-2}| \ge 2t^{1/2} \sigma_{n-1}^{-1} \sqrt{\sum_{j=1}^{n} \sigma_j^{-2}}+t \sigma_{n}^{-2} \right)  \le 2 \exp(-C' \frac{t}{K_0^4} ).$$

 Thus, with probability at least $1- 2\exp(-C' \frac{t}{K_0^4})$,  we have
 
 \begin{equation}\label{eqn:weighted:2}
 \sum_{j=1}^{n} \sigma_j^{-2} ( \col_1^T \Bu_j)^2 \le  \sum_{j=1}^{n} \sigma_j^{-2} + 2t^{1/2} \sigma_{n}^{-1} \sqrt{\sum_{j=1}^{n} \sigma_j^{-2}}+t \sigma_{n}^{-2}.
 \end{equation} 
 
Now we can conclude  from \eqref{eqn:weighted:1}, \eqref{eqn:weighted:2} and Claim \ref{claim:non-weight} (noting that $t\ge C\log^{3/2}n$) that with probability at least $1 -2\exp ( -C'\frac{t}{K_0^4})$ 
 
 $$\frac{n}{t} \le (\frac{n}{2t} + 100t \sigma_{n}^{-2})+  2t^{1/2} \sigma_{n}^{-1}  \sqrt{\frac{n}{2t} + 100t \sigma_{n}^{-2}} +t \sigma_{n}^{-2}. $$ 

This event guarantees that  $\frac{n}{4t} \le 100 t \sigma_{n}^{-2}$, or equivalently  $\sigma_{n} \le 20\frac{t}{\sqrt{n}}$. Our proof is complete.

\section{Normality of vectors: proof of \eqref{eqn:normality}}\label{section:normality} 

We will show that

\begin{equation}\label{eqn:d=1}
|\P((\sqrt{n}x_1\in \Omega) - \P(\Bg_{\F,1} \in \Omega)| \le n^{-c'}.
\end{equation}

The general case with joint distribution of $d$ components, with $d$ chosen to be a small power of $n$, can be treated similarly; see also \eqref{eqn:d=d} below.

Our method follows  that of \cite{TVleast}.   First, by \eqref{eqn:infty:1} of Theorem \ref{thm:linear}, it suffices to work with the event $\CE$

\begin{equation}\label{eqn:CE}
|x_i| = O(\frac{\log^{1/2}n}{\sqrt{n}}), 1\le i\le n.
\end{equation}

We will need the following result  (see for instance \cite[Theorem 3.1]{ACW}).

\begin{theorem}\label{thm:concentration} Assume that $\sum_{i=1}^n  n|x_i|^2=n$ and $n|x_i|^2 \le L$ for all $i$. Then there exists an absolute constant $c$ such that for a uniformly randomly chosen $(m-1)$-set $\{i_1,\dots,i_{m-1}\}$ from the index set $\{2,\dots,n\}$

$$\P\Big(\big | n|x_{i_1}|^2 +\dots + n|x_{i_{m-1}}|^2 - (m-1)\big | \ge t\Big) \le 2\exp(-ct^2/L^2),$$
where the probability is with respect to $\{i_1,\dots,i_{m-1}\}$.
\end{theorem}

For convenience, denote by  $\CF_{i_1,\dots,i_{m-1}}$  the event 

\begin{equation}\label{eqn:CF}
\Big |n|x_{i_1}|^2 +\dots + n|x_{i_{m-1}}|^2 - (m-1)\Big | \le \log^2 n.
\end{equation}

By Theorem \ref{thm:concentration}, with $L=O(\log^{1/2}n)$ and $t=\log^2 n$

$$\P(\CE \wedge \CF_{i_1,\dots,i_{m-1}}) =1 -n^{-\omega(1)}.$$

We are conditioning on these two events for the rest of the argument. 

With foresight, we  choose $m$ slightly larger than the value in Section \ref{section:infty}, in particularly $m$ will take the form $n^{1/C_0}$ for some sufficiently large constant $C_0$ to be chosen later.
We next exploit \eqref{eqn:starting} once more by projecting onto the orthogonal complement $H^\perp$ of $H$. This time we view the projection as $\pi_H:\R^{n-1} \to \R^{m-1}$,  

$$x_{1} \pi_H(\col_1)+ x_{i_1} \pi_H(\col_{i_1}) +\dots + x_{i_{m-1}} \pi_H(\col_{i_{m-1}}) = \sum_{j=0}^{m-1} x_{i_j} \pi_H(\col_{i_j}) = 0.$$

By a normalization $y_{i_j}:= x_{i_j} /\sqrt{\sum_{j=0}^{m-1} |x_{i_j}|^2}$, we rewrite as (with $i_0=1$)

\begin{equation}
y_{i_0} \pi_H(\col_{i_0})+ y_{i_1} \pi_H(\col_{i_1}) +\dots + y_{i_{m-1}} \pi_H(\col_{i_{m-1}}) = 0.
\end{equation}

For $1\le i\le n-1$, let $\Bu_i=\pi_H(\Be_i)\in \R^{m-1}$ be the projection of the standard unit vector $\Be_i$. Then for $1\le j\le m-1$

$$\pi_H(\col_{i_j}) = \sum_{1\le i\le n-1} a_{ii_j} \Bu_i,$$

where $a_{ii_j}$ are the entries of our matrix $A$.

In other words, one can view the $(m-1)\times m$ matrix $M=(\pi_H(\col_{i_0}), \pi_H(\col_{i_1})\dots, \pi(\col_{i_{m-1}}))$ as

$$M = \sum_{1\le i\le n-1,0\le j\le m-1} a_{ii_j} M_{ii_j},$$

where $M_{ii_j}$ is the $(m-1) \times m$ matrix whose columns are zero except the $i_j$-th one, which is $\Bu_i$. Next we record a useful lemma about the matrix $M$, which can be proved by standard techniques from \cite{RV-LO, TVcomp}. 

\begin{lemma}\label{lemma:M}
With high probability with respect to $a_{ij}$, the least singular value of $M M^\ast$ is at least $m^{-2}$ and the largest singular value of $M M^\ast$ is at most $m^2$. 
\end{lemma}


Let $\CE=\CE_{i_1,\dots,i_{m-1}}$ be this event. By the property of projection $\pi_H^\ast \pi_H=I_{m-1}$, 

\begin{equation}\label{eqn:I}
\sum_{0\le j\le m-1} \sum_{1\le i\le n-1} M_{ii_j} M^\ast_{ii_j} =I_{m(m-1)},
\end{equation}

where  we view $M_{ii_j}$ as vectors in $\R^{m(m-1)}$. 

Now for any fixed $(t_0,\dots,t_d) \in  {\R_+}^{d+1}$, with $d\le m^{1/2}$, let $\Omega \subset \R^{m(m-1)}$ be the set of matrices $M$ of size $m\times (m-1)$ satisfying Lemma \ref{lemma:M} such that the normal vector $(y_0,\dots,y_{m-1})$ satisfies $\sqrt{m}|y_{i_0}| \le t_0,\dots, \sqrt{m}|y_{i_d}|\le t_d$. For convenience, define

\begin{align}\label{eqn:pt}
p_{t_0,\dots,t_d}&:= \P\Big(\sqrt{m} |y_{i_0}| \le t_0,\dots, \sqrt{m} |y_{i_d}| \le t_d \Big |\col_j, j\notin \{i_0,\dots,i_{m-1}\} \wedge \CE\Big) \nonumber \\
&= \P\Big( \sum_{1\le i\le n-1,0\le j\le m-1} a_{ii_j} M_{ii_j}\in \Omega\Big | \col_j, j\notin \{i_0,\dots,i_{m-1}\} \wedge \CE \Big).
\end{align}

As with \eqref{eqn:I} we are ready to apply Lemma \ref{lemma:BE}. It is crucial to notice that conditioning on $ \col_j, j\notin \{i_0,\dots,i_{m-1}\}$, the approximating matrix $\sum_{1\le i\le n-1,0\le j\le m-1} \Bg_{ii_j} M_{ii_j}$ is a gaussian iid matrix of size $(m-1)\times m$, and hence Theorem \ref{thm:gaussian} applies to the normal vector $(y_{i_0,\Bg},\dots,y_{i_d,\Bg})$ of this matrix

\begin{align}\label{eqn:comparison:0}
& \P\Big( \sum_{1\le i\le n-1,0\le j\le m-1} \Bg_{ii_j} M_{ii_j} \in  \Omega/ \partial_\eps \Omega  \Big) - O\Big(m^{5} \eps^{-3} \max_{ii_j} \|M_{ii_j}\|_\infty\Big) \le p_{t_0,\dots,t_d} \le \nonumber \\
& \le  \P\Big( \sum_{1\le i\le n-1,0\le j\le m-1} \Bg_{ii_j} M_{ii_j} \in \Omega \cup \partial_\eps \Omega  \Big) +O\Big(m^{5} \eps^{-3}  \max_{ij} \|M_{ij}\|_\infty \Big).
\end{align}

For $\|M_{ii_j}\|_{\infty}$, we apply the following crucial lemma from \cite[Proposition 3.5]{TVleast}.

\begin{lemma}[flatness of orthogonal projection]\label{lemma:norm} There exists a positive constant $c$ (independently of $C_0$) such that the following holds with overwhelming probability with respect to $ \col_j, j\notin \{i_0,\dots,i_{m-1}\}$: for any unit vector $\Bv\in H^\perp$ we have

$$\|\Bv\|_\infty \le n^{-c}.$$
\end{lemma}

For short we let $\CG_{i_1,\dots,i_{m-1}}$ be the event considered in Lemma \ref{lemma:norm}, thus 

$$\P(\CG_{i_1,\dots,i_{m-1}})=1-n^{-\omega(1)}.$$ 

Let us now consider the sets $ \Omega/ \partial_\eps \Omega $ and $\Omega \cup \partial_\eps \Omega$. Assume that $M,M' \in \Omega$ with normal vectors $\By=(y_0,\dots,y_{m-1})$ and $\By'=(y_0',\dots,y_{m-1}')$ and such that $\|M-M'\|_\infty \le \eps$. Then as $\|M \By'\|_2 = \|(M-M')\By'\|_2 \le \|M'-M\|_2 \le m \eps$, we have (rather generously) $\|M^\ast M \By'\|_2 \le m^3 \eps$. By definition of $\Omega$ (which satisfies Lemma \ref{lemma:M}), it then follows that (again very generously)

$$\|\By-\By\|_\infty \le m^8 \eps.$$

Hence it follows from \eqref{eqn:comparison:0} that

\begin{align}\label{eqn:comparison:1}
&\P\Big(\sqrt{m}|y_{i_0,\Bg}| \le t_0-m^8 \eps,\dots, \sqrt{m}|y_{i_d,\Bg}| \le t_d- m^8 \eps \Big) - O\Big(m^{5} \eps^{-3} \max_{ii_j} \|M_{ii_j}\|_\infty\Big) \le p_{t_0,\dots,t_d} \le \nonumber \\
& \le  \P\Big(\sqrt{m}|y_{i_0,\Bg}| \le t_0+m^8 \eps,\dots, \sqrt{m}|y_{i_d,\Bg}| \le t_d + m^8 \eps \Big) +O\Big(m^{5} \eps^{-3}  \max_{ij} \|M_{ij}\|_\infty \Big).
\end{align}

Now choose  $\eps=n^{-c/4}$ (with $c$ from Lemma \ref{lemma:norm}) and $m=n^{c/64}$. We have 

\begin{align}\label{eqn:comparison:1'}
&\P\Big(\sqrt{m}|y_{i_0,\Bg}| \le t_0-n^{-c/8},\dots, \sqrt{m}|y_{i_d,\Bg}| \le t_d-n^{-c/8} \Big) - O(n^{-c/8}) \le p_{t_0,\dots,t_d} \le \nonumber \\ 
&\le \P\Big(\sqrt{m}|y_{i_0,\Bg}| \le t_0+n^{-c/8},\dots, \sqrt{m}|y_{i_d,\Bg}| \le t_d+n^{-c/8} \Big)  +O(n^{-c/8}).
\end{align}

By Theorem \ref{thm:gaussian}, we have, for some constant $c'$ sufficiently small depending on $c$

\begin{equation}\label{eqn:comparison:2}
\Big |p_{t_0,\dots,t_d}- \P\Big(\sqrt{m}|y_{i_0,\Bg}| \le t_0,\dots, \sqrt{m}|y_{i_d,\Bg}| \le t_d \Big) \Big |= O(n^{-c'}).
\end{equation}

Now we pass from $\sqrt{m}y_{i_j}$ to $\sqrt{n}x_{i_j}$ conditioning on $\CE \wedge \CF_{i_1,\dots,i_{m-1}}$. On this event, by \eqref{eqn:CF}
$$\left| |x_{i_0}|^2+\dots+|x_{i_{m-1}}|^2 - \frac{m}{n}\right | \le \frac{\log^2 n}{n}.$$ 

In other words,

$$\Big|\sqrt{|x_{i_0}|^2+\dots+|x_{i_{m-1}}|^2} - \sqrt{\frac{m}{n}} \Big| \ll \frac{\log^2n}{n} / \sqrt{\frac{m}{n}} \ll  \frac{\log^2 n}{\sqrt{ mn}}.$$

Consequently, 

\begin{align}\label{eqn:comparison:3}
|\sqrt{m}y_{i_j} - \sqrt{n}x_{i_j}|&= \Big|\sqrt{m} x_{i_j} ( \frac{1}{\sqrt{|x_{i_0}|^2+\dots+|x_{i_{m-1}}|^2}} - \frac{1}{\sqrt{\frac{m}{n}}}) \Big| \nonumber \\
& \ll |\sqrt{m} x_{i_j}| \frac{\log^2 n}{\sqrt{mn}}/ \frac{m}{n} \nonumber \\
&\ll \frac{\log^{5/2} n}{m},
\end{align}

where we used the bound $|x_i|=O( \frac{\log^{1/2}n}{\sqrt{n}})$ in the last estimate.

In summary, it follows from \eqref{eqn:comparison:2} and \eqref{eqn:comparison:3} that conditioning on $\CE \wedge \CF_{i_1,\dots,i_{m-1}}\wedge \CG_{i_1,\dots,i_{m-1}}$ 

\begin{equation}\label{eqn:d=d}
\P\Big(\sqrt{n}|x_{i_0}|\le t_0,\dots, \sqrt{n} |x_{i_d}|\le t_d\Big) =\P\Big(|\Bg_{\F,0}|\le t_0,\dots, |\Bg_{\F,d}|\le t_d\Big) +O(n^{-c''}).
\end{equation}

for some absolute constant $c''$. In particularly, this immediately implies \eqref{eqn:d=1} as all of the conditional events hold with high probability.


\section{proof of \eqref{eqn:innerproduct}}\label{section:innerproduct}

The  treatment here follows closely  \cite[Proposition 25]{TVvector}. Let $\alpha$ be a number growing slowly to infinity that will be specified later. For each component $x_i$ of $\Bx$ we decompose

$$x_i = x_i 1_{|\sqrt{n}x_i| \le \alpha} + x_i 1_{|\sqrt{n}x_i > \alpha}:= x_{i, \le} + x_{i,>}.$$ 

We then decompose $\Bx = \Bx_{\le} + \Bx_{>}$ accordingly. For \eqref{eqn:innerproduct} it suffices to show 
\begin{claim} With an appropriate choice of $\alpha$ we have
\begin{enumerate}[(i)]
\item $\sqrt{n}\Bx_{\le}^T \Bu \overset{d}{\to} N(0,1)$;
\vskip .1in
\item $\sqrt{n}\Bx_{>}^T \Bu$ converges to zero in probability.
\end{enumerate}
\end{claim}

For (ii), we will estimate the second moment

$$n \E(\Bx_{>}^T \Bu)^2 = n \sum_{1\le i\le n}\sum_{1\le j\le n} u_iu_j \E x_{i,>} x_{j,>} .
 $$
 
 Because $\xi$ is symmetric, $\E  x_{i,>} x_{j,>} =0$ if $i\neq j$, and hence 
 
 $$n \E(\Bx_{>}^T \Bu)^2 = n \sum_{1\le i\le n} u_i^2 \E x_{i,>}^2 = n \E x_{1,>}^2.$$
 
 Now, by the exchangeability, $n\E x_1^2=1$. Also, by \eqref{eqn:normality}
 
 $$n \E x_1^2 1_{|\sqrt{n}x_1| \le \alpha} = \E N(0,1)^2 1_{|N(0,1)|\le \alpha} + O(n^{-c}).$$
 
 It thus follows that, as $\alpha \to \infty$ together with $n$ 
 
 \begin{equation}\label{eqn:2nd}
 n \E x_{1,>}^2= n \E x_1^2 1_{|\sqrt{n}x_1| > \alpha} = o(1).
 \end{equation} 
 
By Markov's bound,  $|\sqrt{n}\Bx_{>}^T \Bu| \to 0$ in probability as claimed in (ii). 
 
 For (i), by Carleman's criteria, it suffices to show that for every fix positive integer $k$ the $k$-moment of $\sqrt{n}\Bx_{\le}^T \Bu$ asymptotically matches with that of $\Bg_\R$. We have
 
 $$n^{k/2} \E (\Bx_{\le}^T \Bu)^k = n^{k/2} \sum_{i_1,\dots,i_k} u_{i_1}\dots u_{i_k} \E x_{i_1,\le }\dots x_{i_k,\le }.$$
  
  Now we make use of the symmetry assumption on $\xi$. By this assumption, 
   the expectation vanishes unless each index $i$ appears an even number of times. Furthermore, by \eqref{eqn:normality}

  $$n^{k/2} \E x_{i_1,\le }\dots x_{i_k,\le } = \E \Bg_{\R,{i_1}} \dots \Bg_{\R,i_k} + o(1).$$
  
  Thus 
  
 \begin{align*}
  n^{k/2} \sum_{i_1,\dots,i_k} u_{i_1}\dots u_{i_k} \E x_{i_1,\le }\dots x_{i_k,\le } &=\sum_{i_1,\dots,i_k} u_{i_1}\dots u_{i_k}   \E \Bg_{\R,{i_1}} \dots \Bg_{\R,i_k} + o(\sum_{*} |u_{i_1} \dots u_{i_k}|)\\
 &= \E (\Bg_\R)^k +  o(\sum_{*}|u_{i_1} \dots u_{i_k}|),
 \end{align*}
  
where the implied constant can depend on $k$, and $\sum_{*}$ indicates over all $k$-tuples $i_1,\dots,i_k$ in which each index $i$ appears an even number of times.  To complete the proof, we just note that 
 
 $$\sum_{*}|u_{i_1} \dots u_{i_k}| \le (\sum_i u_i^2)^{k/2}=1.$$

\vskip .2in

{\bf Acknowledgements.} The authors are thankful to K.~Wang for helpful discussion. They are also grateful to A.~Knowles with help of references.

\end{document}